\documentclass[12pt]{amsart}
\usepackage{amssymb,amsmath,MnSymbol}
\usepackage[english]{babel}
\usepackage{amsfonts}
\usepackage{color}
\usepackage{amsthm}
\usepackage[colorlinks=true,linkcolor=NavyBlue, citecolor=NavyBlue]{hyperref}
\usepackage{graphicx}
\usepackage{color}
\usepackage{mathtools}
\usepackage[usenames,dvipsnames]{pstricks}
\usepackage{bm}
\usepackage{amsmath}

\DeclareMathOperator{\AGL}{AGL}
\DeclareMathOperator{\GL}{GL}

\DeclareMathOperator{\Sym}{Sym}

\DeclareMathOperator{\F}{\mathbb{F}}
\DeclareMathOperator{\Alt}{Alt}
\def\FF{\F_2}

\newcommand\deq{\mathrel{\stackrel{\makebox[0pt]{\mbox{\normalfont\tiny def}}}{=}}}
\newcommand{\ve}{\varepsilon}
\newcommand{\Size}[1]{\left|\,#1\,\right|}
\newcommand{\Span}[1]{\left\langle\,#1\,\right\rangle}
\newcommand{\Set}[1]{\left\{\,#1\,\right\}}
\newcommand{\p}{\circ}
\newcommand{\isnorm}{\trianglelefteq}

\newcommand{\one}{1}
\let\phi\varphi

\theoremstyle{plain}
\newtheorem{theorem}{Theorem}

\newtheorem{lemma}[theorem]{Lemma}
\newtheorem{proposition}[theorem]{Proposition}
\newtheorem{corollary}[theorem]{Corollary}

\theoremstyle{remark}
\newtheorem{remark}[theorem]{Remark}
\theoremstyle{definition}

\newtheorem*{notation*}{Notation}

\begin{document}
\title{Regular subgroups with large intersection}
\author{R. Aragona} 
\address[Riccardo Aragona]%
{Dipartimento di Ingegneria e Scienze dell'Informazione e Ma\-te\-ma\-ti\-ca\\
Universit\`a degli Studi dell'Aquila\\
 Via Vetoio\\
 I-67100 Coppito (L'Aquila)\\
Italy}
\email{riccardo.aragona@univaq.it} 
\author{R. Civino} 
\address[Roberto Civino]%
{Dipartimento di Ingegneria e Scienze dell'Informazione e Ma\-te\-ma\-ti\-ca\\
Universit\`a degli Studi dell'Aquila\\
 Via Vetoio\\
 I-67100 Coppito (L'Aquila)\\
Italy}
\email{roberto.civino@univaq.it} 
\author{N. Gavioli} 
\address[Norberto Gavioli]%
{Dipartimento di Ingegneria e Scienze dell'Informazione e Ma\-te\-ma\-ti\-ca\\
Universit\`a degli Studi dell'Aquila\\
 Via Vetoio\\
 I-67100 Coppito (L'Aquila)\\
Italy}
\email{norberto.gavioli@univaq.it} 
\author{C. M. Scoppola}
\address[Carlo Maria Scoppola]%
{Dipartimento di Ingegneria e Scienze dell'Informazione e Ma\-te\-ma\-ti\-ca\\
Universit\`a degli Studi dell'Aquila\\
 Via Vetoio\\
 I-67100 Coppito (L'Aquila)\\
Italy}
\email{scoppola@univaq.it} 
\date{}
\thanks{R. Aragona, N. Gavioli, and C. M. Scoppola are members of INdAM-GNSAGA (Italy). R. Civino thankfully acknowledges support by the Department of Mathematics of the University of Trento. The authors thankfully acknowledge support by MIUR-Italy via PRIN 2015TW9LSR ``Group theory and applications''. Part of this work has been carried out during the cycle of seminars ``Gruppi al Centro'' organised at INdAM in Rome}

\maketitle
\begin{abstract}
 In this paper we study the relationships between the elementary abelian regular subgroups and the Sylow $2$-subgroups of their normalisers in the symmetric group $\Sym(\FF^n)$, in view of the interest that they have recently raised for their applications in symmetric cryptography.
\end{abstract}


\section{Introduction}
Let $n >2$ and let $(V,+)$ be an $n$-dimensional vector space over the field with two elements, where  $+$  denotes the bitwise XOR operation, i.e.\ the bitwise addition modulo two.
The conjugacy class of elementary abelian regular subgroups of the symmetric group $\Sym(V)$ has recently drawn the attention of symmetric cryptographers, as these subgroups and their normalisers may be used to detect weaknesses in symmetric-encryption methods, i.e.\ block ciphers. More specifically, cryptanalysts may take advantage of the alternative operations that these groups induce on the plaintext space and exploit them  to detect biases in the distribution of the ciphertexts. 

\medskip
In this paper, motivated by the possible cryptographic applications, we consider the families of  \emph{maximal-intersection subgroups} and \emph{second-maximal-intersection subgroups} of $\Sym(V)$, i.e.\ the families of elementary abelian regular subgroups of $\Sym(V)$ that intersect the image $\sigma_V$ of the right regular representation $\sigma$, here usually denoted by $T$, in a subgroup of index two or four in $T$. We prove that each second-maximal-intersection subgroup is affine. In other words, such a subgroup is contained in $\AGL(V)$, the normaliser of $T$ in $\Sym(V)$, which is a maximal subgroup of the alternating group $\Alt(V)$~\cite{praeger}. Moreover, we prove that every Sylow $2$-subgroup of $\AGL(V)$ contains one and only one second-maximal-intersection subgroup as a normal subgroup. As a consequence, we conclude that $[N_{\Sym(V)}(\Sigma):\Sigma]=2$, where $\Sigma$ is a Sylow $2$-subgroup of $\AGL(V)$.

\subsection*{Motivation and links to symmetric cryptography}
 A block cipher on the plaintext space $V$ is a family $\Set{E_k}_{k\in\mathcal K}$ of non-linear permutations of $\Sym(V)$, called \emph{encryption functions}, indexed by a set of parameters $\mathcal K$, called \emph{keys}. Each encryption function is usually obtained as the composition of different layers, each one designed with a precise cryptographic goal, depending on its role in the employed algorithm (see e.g.~\cite{aes,present,des}).
Some of those layers provide entropy to the encryption process by additions with \emph{round keys} in $V$ computed by a public procedure, called \emph{key schedule},  starting from the user-selected key in $\mathcal K$. The non-linearity of the functions $E_k$ is one of the crucial requirements to provide security against a large variety of statistical attacks, such as differential~\cite{bih91} and linear~\cite{mat93} cryptanalysis. 
For this reason, ways of making the cipher's components as far as possible from being linear are extensively studied~\cite{nyb93}. The usually considered notion of non-linearity is given with respect to the operation which is used in the cipher to perform the key addition. However
it is worth noticing here that the notion of non-linearity is not univocal. For example, one of the classical notions of non-linearity for $f \in \Sym(V)$ is the one that measures the distance of $f$ from the set of the affine functions $\AGL(V)$~\cite{carlet10}. Another well-established definition~\cite{nyb93} looks at the behaviour of the derivatives of $f$, measuring how far they are from being constant. Other notions of non-linearity may be found in~\cite{canteaut10, aragona17}. As already mentioned, the security of a cipher depends, among other things, on the requirement that its encryption functions do not behave as linear functions, i.e.\ they lie far from the set $\AGL(V)$. However, several isomorphic copies of $\AGL(V)$ are contained in $\Sym(V)$, and each of them corresponds to a different operation endowing $V$ with a distinct vector space structure. A target of a new branch of research in symmetric cryptography~\cite{cds06, Calderini2017, Brunetta2017, Civino2018} is to investigate the non-linearity of the encryption functions of a cipher with respect to these alternative operations.\\ 

Let us recall that $\sigma \,\colon V\to \Sym(V)$ denotes the right regular representation and that
\[T \,\deq \sigma_V=\Set{\sigma_v \mid v\in V,\,  x \mapsto x+v}.\] 
If  $\tau\,\colon v\mapsto \tau_v$ is another embedding of $V$ in $\Sym(V)$ as a regular permutation subgroup, then we denote by 
$\tau_V = \{\tau_v \mid v \in V\}$ its image, where the map ${\tau_v}$ is the one for which $0 \mapsto v$. A new operation $\circ$ on $V$ may be defined from $\tau_V$ by setting 
\[\forall u, v \in V \quad u\circ v \deq u\tau_v.\]
 It is straightforward to check that $(V,\circ)$ is an elementary abelian $2$-group. The notation $T_\p$ for $\tau_V$, used in~\cite{Calderini2017, Brunetta2017, Civino2018},  might be ambiguous in this paper. For this reason, we prefer to use the more explicit notation $\tau_V = T^g$, where the element $g \in \Sym(V)$ conjugates $T$ in $\tau_V$. The existence of such an element is a consequence of a result by Dixon~\cite{Dixon1971}, that we recall in Section~\ref{sec_pre}.\\
 
In what follows we quickly describe the contributions of~\cite{cds06, Calderini2017, Civino2018,Brunetta2017} in order to give an idea of the possible attacks.\\
Abelian regular  subgroups of the affine group $\AGL(V)$  are described in~\cite{cds06} in terms of commutative associative algebra structures  defined on $V$. This is also the case for $T^g$. In~\cite{Calderini2017,Brunetta2017} the authors designed a toy cipher whose set of encryption functions is contained in a conjugate  $\AGL(V)^g$ for some $g\in \Sym(V)$. In other words, the encryption functions are affine with respect to the new operation, different from the classical bitwise XOR, defined as above from  $T^g$. 
The first differential attack~\cite{bih91} using an alternative operation has been performed in~\cite{Civino2018}, where the authors designed a cipher which is resistant to the classical differential attack with respect to $+$ but may be attacked using another operation specifically created.

\noindent In the works~\cite{Calderini2017, Civino2018, Brunetta2017} the subspace 
\[
W_\circ \deq \Set{k \in V \mid \forall x \in V \quad k\p x = k+x}
\]
plays an important cryptographic role, since it represents the set of round keys for which the XOR addition and the $\p$-addition give the same result. For this reason they are called \emph{weak keys} and $W_\p$ is called the \emph{weak-key subspace}. The same notation is also used in this paper. 
 It is straightforward to check that $W_\p$ is a subspace of both $(V,+)$ and $(V,\p)$ and $\sigma_{W_\circ}=\tau_{W_\circ}=T \cap T_\p$.\\

Finally, it is worth mentioning that the cryptanalysis exploiting an operation  different from the one used to perform the key addition may be a hard task, since usually the transformations which are affine  with respect to the new operation  may not be affine with respect to the classical one. For this reason, some modern-cipher  designers decided to alternate in their algorithms  several  layers, each of which is affine with respect to a different operation. A classic example of this design strategy is the Russian government standard GOST~\cite{dolmatov2010gost}, where an addition modulo $2^{32}$ is used besides the classical XOR for the key addition. Such a design strategy makes the use of standard cryptanalytic techniques more difficult. For this reason, in the case of GOST, only few results are known regarding the cryptanalysis, see e.g.~\cite{seki2000differential,aragona2017group}.
\subsection*{Organisation of the paper}
In Sec.~\ref{sec_pre} we introduce our notation and we recall some known preliminary results. Sec.~\ref{sec:regular} is dedicated to the study of the
maximal-intersection subgroups and second-maximal-intersection subgroups of $\Sym(V)$.
In our main results (Theorems~\ref{thm_MaxSub} and~\ref{thm_TinAGL}) 
we parametrise such groups. In Sec.~\ref{sec_syl} we focus our attention on the case of second-maximal-intersection subgroups and on the Sylow $2$-subgroups of $\AGL(V)$. We show that every Sylow $2$-subgroup of $\AGL(V)$ contains one and only one second-maximal-intersection subgroup as a normal subgroup (Theorem~\ref{prop_Syl}). As a consequence, we show that every Sylow $2$-subgroup of $\AGL(V)$ is self-normalising in $\AGL(V)$ and has index $2$ in its normaliser in $\Sym(V)$ (Theorem~\ref{thm:N=S}). Lastly, Sec.~\ref{sec_concl} concludes the paper with some open problems.


\section{Notation and preliminary results}\label{sec_pre}
 We have already introduced part of our (rather standard) notation. Moreover, the set $\Set{e_1,e_2,\ldots, e_n}$ denotes the canonical basis of $V$. 
For each given vector $v \in V$ we denote by $v^{(i)} \in \FF$ the $i$-th coordinate of $v$, and by $v^{(i:j)} \in \FF^{j-i+1}$ the vector composed by the coordinates of $v$ from the $i$-th to the $j$-th, for $1\leq i<j \leq n$.
If $G$ is a group acting on $V$ we denote by $vg$ the image of the action of $g \in G$ on $v \in V$.  The identity element of $G$ is denoted by $\one_G$. 
We also recall that the affine group $\AGL(V)$ is $T \rtimes \GL(V)$, where $\GL(V)$ is the group of linear bijections of $V$. The identity matrix of $\GL\left(\FF^{\,d}\right)$ is also denoted by $\one_d$.  \\

Let  now $\tau$ be a regular representation as in the previous section.
In this paper we will extensively use the fact that, by ~\cite[proof of Lemma 1]{Dixon1971}, there exists $g\in \Sym(V)$ such that  $\tau_V=T^g$. Denoting by $\tau_v$ the unique map of $\tau_V$ sending $0$  to $v$,
one has $(\sigma_{v})^{g}=\tau_{(0g^{-1}+v)g}$ for all $v\in V$. For the convenience of the reader the proof of Dixon's result is reproduced here. 
\begin{theorem}
Let $X = \Set{x_1,x_2,\ldots, x_m}$ be a finite set and let $H$ and $K$ be regular subgroups of $\Sym(X)$. If $H \cong K$, then there exists $g \in \Sym(X)$ such that $K = H^g$.
\end{theorem}
\begin{proof}
Let $\zeta: H \to K$ be an isomorphism. Since both the groups are regular,
we have that $\Set{x_1h \mid h \in H} = \Set{x_1k \mid  k \in K} = X$. Let us now define the permutation $g\in \Sym(X)$ by setting $(x_1 h)g\deq x_1(h\zeta)$, for each $h \in H$. The result follows from the fact that $h^g=h\zeta$ for all $h \in H$. Indeed, let $1 \leq i \leq m$ and let $h'\zeta \in K$ such that $x_1h'\zeta=x_i$, where $h' \in H$. Then, since $\zeta$ is an isomorphism we have
\[
x_i h^g=x_ig^{-1}hg=(x_1h'\zeta) g^{-1}hg=(x_1h'h)g = x_1(h' h \zeta) =x_1h'\zeta h \zeta  = x_i h\zeta.
\]
Since the property holds for each $1 \leq i \leq m$, then for each $h \in H$ we have $h^g = h\zeta$, therefore $H^g = H\zeta = K$.
\end{proof}
The following lemma generalises a well know fact.
\begin{lemma}\label{lem_Tnorm}
Every elementary abelian regular subgroup of $\Sym(V)$ is the unique non-trivial proper normal subgroup of its own normaliser.  
\end{lemma}
\begin{proof}
By~\cite{Dixon1971}, the result may be proven for the regular group $T$, up to conjugation.
Since $T$ is well known to be a self-centralising minimal normal subgroup of $\AGL(V)$, if $\bar{T} \neq T$ is a minimal normal subgroup of $\AGL(V)$, then $T\cap \bar{T}=\Set{\one_{\Sym(V)}}$. So that $\bar{T}$ centralises $T$, from which it follows that $\bar{T} < T$, a contradiction. 
Now, let $N$ be a non-trivial normal subgroup of $\AGL(V)$. By the previous argument, $T \leq N$ and $N / T \isnorm \AGL(V) / T \cong \GL(V)$, which is simple since $\dim(V) > 2$ and the ground field has characteristic 2. Therefore $N = \AGL(V)$.
\end{proof}

The following remark will be useful to describe the centraliser over $\Sym(V)$ of a subgroup of $T$. 
\begin{remark}\label{rem:cent-wr}
Let $H$ be a group acting transitively on a set $X$ and let $Z\deq X\times Y$, where $Y$ is a set.  There exists a canonical embedding $\theta:H\to\Sym(Z)$ under which the action of $H$ is defined by  $(x,y)h^\theta=(xh,y)$. Besides, there exists another embedding 
\[
\Sym(X)^Y \hookrightarrow \Sym(Z)
\]
defined by $(x,y)f=(xf_y,y)$, where $f\in\Sym(X)^Y$ is the function sending $y\mapsto f_y\in\Sym(X)$. In particular $C_{\Sym(X)}(H)^Y$ is a subgroup of $\Sym(Z)$ which centralises $H^\theta$.
Notice that also $\Sym(Y)$ embeds in $\Sym(Z)$ by way of $(x,y)\pi=(x,y\pi)$, for $\pi\in\Sym(Y)$. This group centralises $H^\theta$, and normalises and intersects trivially $C_{\Sym(X)}(H)^Y$. As a consequence we have
\[
C_{\Sym(X)}(H)^Y\rtimes \Sym(Y) = C_{\Sym(X)}(H)\wr_{Y} \Sym(Y)\leq C_{\Sym(Z)}(H^{\theta}).
\]
It is straightforward, but somewhat lengthy, to show that the opposite inclusion holds, i.e.
\begin{equation}\label{eq:wr}
C_{\Sym(Z)}(H^{\theta})=C_{\Sym(X)}(H)\wr_{Y} \Sym(Y).
\end{equation} 
\end{remark}
As a consequence of the previous remark we can prove the following.
\begin{lemma}\label{lem_Cwreath}
Let $G$ be a finite group and $\sigma,\lambda: G \to \Sym(G)$ respectively be the  right and left regular representations of $G$. If $H\leq G$, then 
\[
C_{\Sym(G)}(H^\sigma)=H^\lambda \wr \Sym(G/H),
\]
where $G/H$ is the set of left cosets of $H$ in $G$.
\end{lemma}
\begin{proof}
It is well known that $C_{\Sym(H)}(H^\sigma)=H^\lambda$, therefore the claim follows from Eq.~\eqref{eq:wr}.
\end{proof}

\begin{corollary}\label{cor_Cwreath}
If $M$ is a subgroup of $T$ of order $2^{n-m}$, then $C_{\Sym(V)}(M)$ is the wreath product $M\wr \Sym(2^m)$. In particular $\Size{C_{\Sym(V)}(M)} = 2^m!\,2^{2^{m}(n-m)}$. 
\end{corollary}
\begin{proof}
The result follows from Lemma \ref{lem_Cwreath}. Indeed, since $T$ is abelian, $M^\sigma=M^\lambda=M$.
\end{proof}


\section{Elementary abelian regular subgroups whose intersection with the translation group is large}\label{sec:regular}
In this section we prove some results on elementary abelian regular subgroups of $\Sym(V)$, and more generally on fixed-point-free involutions. 
Our interest is in particular in the connections between such groups and $\AGL(V)$. \\

We now parametrise the elementary abelian regular subgroups of $\Sym(V)$ according to the size of their intersection with $T$.
\subsection{Maximal intersection}
Here we prove that none of the aforementioned groups has a maximal intersection with $T$. The result is a consequence of the following proposition, which slightly generalises a result appearing in~\cite{Calderini2017}.
\begin{proposition}\label{six}
Let $g \in \Sym(V)$ such that $T \neq T^g$. If $W\leq V$ such that $\sigma_W =T \cap T^g$, then $\dim(W) \leq n-2$.
\end{proposition}
\begin{proof}
Assume by way of contradiction that $\dim(W)=n-1$.
Let $\Set{v_i}_{i=1}^{n-1}$ be a basis for $W$ and $v \in V\setminus W$. The claim holds if $a\tau_{v} = a \sigma_{v}$ for any $a \in V$. If $a \in W$ there is nothing to prove, hence without loss of generality we may assume $a = w+v$, for some $w \in W$. Then
\begin{align*}
a\tau_{v}&=(w+v)\tau_{v} =   (w\sigma_{v})\tau_{v} \\
 &= (w\tau_{v})\tau_{v} =  w (\tau_{v})^2 = w\\\
 &= a \sigma_{v}. 
 \qedhere
\end{align*}
\end{proof}

From Proposition~\ref{six}, we can derive the following result.
\begin{theorem}\label{thm_MaxSub}
Let $M$ be any maximal subgroup of $T$. Then $C_{\Sym(V)}(M)<\AGL(V)$. Moreover  $\Size{C_{\Sym(V)}(M)}= 2^{2n-1}$.
\end{theorem}
\begin{proof}
It is enough to show that $T \isnorm C_{\Sym(V)}(M)$, so that $C_{\Sym(V)}(M)$ is a subgroup of the normaliser of $T$ which is $\AGL(V)$. To this purpose, let $c \in C_{\Sym(V)}(M)$ and let $\xi$ be any translation in $T \setminus M$. The aim is to show that  $\eta \deq c^{-1}\xi c\in T$. 
Assume the contrary by way of contradiction, so that $\eta \notin T$.  Clearly  $\Span{ M, \eta } = \Span{ M, \xi }^c=T^c\cong T$ and $\Size{\left(T \cap \Span{ M, \eta }\right)}= 2^{n-1}$. This contradicts the previous proposition.
The claim $\Size{C_{\Sym(V)}(M)}= 2^{2n-1}$ follows from Corollary~\ref{cor_Cwreath}.
\end{proof}

\subsection{Second-maximal intersection}
In this section we prove that elementary abelian regular subgroups that intersect $T$ in a second-maximal subgroup are all contained in $\AGL(V)$.\\

From now on we shall assume that $g \in \Sym(V)$ is such that $\sigma_W =T \cap T^g$ and $\dim(W) = n-2$, where $W\leq V$. Let $W_1 \deq W+v_1$, $W_2 \deq W+v_2$ and $W_{1,2} \deq W+v_1+v_2$, for some $v_1$ and $v_2$ in $V$ which are linearly independent modulo $W$, be the non-trivial cosets of $W$ in $V$. Notice that any element  in $\Sym(V)$ centralising $T \cap T^g$ permutes these cosets. \\

In the hypothesis of this section, Theorem~\ref{thm_MaxSub} has not a counterpart. However, we have the following generalisation.
\begin{lemma}\label{lem_cosets}
Let $W \leq V$ be such that $\dim(W) = n-2$.
Let $\phi \in \Sym(V)$ be an involution centralising the second-maximal subgroup $\sigma_W$ of $T$ and also not fixing any of the cosets of $W$. Then $\phi \in \AGL(V)$. 
\end{lemma}
\begin{proof}
Since $\phi$ is regular and it centralises $\sigma_W$, its action on $V$ is completely determined by its action on the cosets of $W$, which we may assume being as the involution $(W,W_1)(W_2, W_{1,2})$, and by the choices of $0\phi$ and $v_2\phi$. Indeed, if $ 0 \ne u \in W$, then $u\phi$ is determined as $u\phi=0\sigma_u\phi=0\phi\sigma_u \in W_1$. Similarly $v_2\phi$ determines the action of $\phi$ on each element of $W_{2}$. Set now $z \deq 0\phi \in W_1$, $x\deq v_1\phi+z+v_1$ and $y\deq v_2\phi+z+v_2$. 
Let $\bar{\phi}$ be the affinity sending $w+\alpha v_1+\beta v_2 \mapsto w+\alpha v_1+\beta v_2 + \alpha x + \beta y +z$, where $\alpha,\beta \in \FF$. Then 
an easy check shows that $0\bar{\phi}=z$, $v_2\bar\phi = v_2\phi$, and that $\bar\phi$ is an involution which centralises $\sigma_W$ and acts on the cosets of $W$ as the involution $(W,W_1)(W_2, W_{1,2})$. Therefore $\bar{\phi}=\phi$, which is what we meant to prove. 
\end{proof}

In the very special case $n=3$ every fixed-point-free involution in $\Sym(V)$ centralising $\sigma_W$ is affine. Indeed, since $n=3$, $W = \Span{w}$. Now, if $\phi$ fixes all the four cosets of $W$, then $\phi$ is the translation $\sigma_w$. If $\phi$ acts on the cosets without fixed points, then Lemma~\ref{lem_cosets} applies. Finally, we are left with the case where two cosets are fixed and two are exchanged. Assuming without loss of generality that $\phi$ 
acts on the cosets as $(W, W_1)$, then $\phi$ is the affinity sending $(\gamma w+\alpha v_1+\beta v_2) \mapsto (\gamma +\alpha) w +(\alpha+\beta+1) v_1 + \beta v_2$, where $\alpha, \beta,\gamma \in \FF$. A straightforward check may be also performed using MAGMA \cite{Bosma1997}.

\begin{remark}\label{rmk_action}
In the following theorem we will use a well-known fact: if $G<\Sym(V)$ is a regular subgroup and $H \isnorm G$, then $G/H$ acts regularly on the set of the orbits of $H$ (i.e.\ the cosets of $H$).
\end{remark}
The following straightforward consequence of Lemma~\ref{lem_cosets} and Remark~\ref{rmk_action} is the main contribution of this section. We however include a second constructive proof, in view of its use in the remainder of the paper. The notation used here and below is the one specified at the beginning of Sec.~\ref{sec_pre}.
\begin{theorem}\label{thm_TinAGL}
Let $g \in \Sym(V)$ and $W\leq V$ such that $\sigma_W =T \cap T^g$. If $\dim(W) = n-2$, then $T^g  < \AGL(V)$.
\end{theorem}
\begin{proof}
It is enough to show that if $T^g = \Span{\pi,\ve, \sigma_W}$, then $\pi, \ve \in \AGL(V)$. This is granted by Lemma~\ref{lem_cosets}, since $\pi$ and $\ve$ are regular involutions centralising $W$ and not fixing any of its cosets. \medskip

Alternatively, we now construct explicitly two affinities $\bar{\pi}$ and $\bar{\ve}$ 
which are respectively congruent to $\pi$ and $\ve$ modulo translations in $\sigma_W$. A similar construction in provided is \cite{Calderini2017}. Let us assume $W$ is spanned by the last $n-2$ vectors of the canonical basis of $V$. Notice that $\pi$, $\ve \in T^g$, thus
\begin{enumerate}
\item $\pi$ and $\ve$ are fixed point free, 
\item $\pi$ and $\ve$ are involutions,
\item $\pi,\ve \in C_{\Sym(V)}(\sigma_W)$,
\item $\pi\ve=\ve\pi$.
\end{enumerate}
Moreover, since $T^g$ is regular, by Remark~\ref{rmk_action} it is then possible to assume that $\pi$ acts on the cosets of $W$ as the involution $(W,W_1)(W_2, W_{1,2})$, whereas $\ve$ acts as $(W,W_2)(W_1, W_{1,2})$. 
Up to a composition by a translation in $\sigma_W$, one can assume that $v_2\pi = v_1+v_2$, so that the action of $\pi$ on $V$ is completely determined by the value of $0\pi \in W_1$.
Similarly, modulo a translation in $\sigma_W$, let us assume that $v_1\ve=v_1+v_2$. Consequently, since $T^g = \Span{\pi, \ve , \sigma_W}$, the action of $T^g$ on $V$ is completely determined by $W$ and by the values of $0\pi$, $0\ve$ and $0\pi\ve$.
What remains to be proven is that for each possibile choice of $0\pi$ in $W_1$,   $0\ve$ in $W_2$ and $0 \pi\ve$ in $W_{1,2}$ there exist two affinities $\bar\pi, \bar\ve 	\in \AGL(V)$ such that for each $v \in V$ we have $v\pi = v\bar\pi $ and $v \ve = v \bar\ve$.  
Consider now the functions $\bar\pi$ sending $x\mapsto x+x^{(2)}b+0\pi$ and $\bar\ve$ sending $x \mapsto x+x^{(1)}b+0\ve$, where $b = 0\pi+0\ve+0\pi\ve$.
Then $\bar\pi, \bar\ve \in \AGL(V)$ satisfy the four properties listed above and 
$0\bar\pi = 0\pi$, $v_2\bar\pi= v_1+v_2 \mod W$, $0\bar\ve = 0\ve$, $v_1\bar\ve= v_1+v_2 \mod W$. Therefore $\pi=\bar{\pi}$ and $\ve = \bar\ve$, hence the desired result is proved. 
\end{proof}

\begin{remark}\label{rem_TinAGLp}
In the hypotheses of Theorem~\ref{thm_TinAGL}, by interchanging the roles of $T$ and $T^g$, one can easily obtain that also $T$ is a subgroup of $\AGL(V)^g$, i.e.\ $T$ normalises $T^g$.
\end{remark}

It is convenient to give a more practical representation of the groups $T^g$ such that $\dim(W )= n-2$, where $\sigma_W = T \cap T^g$.
From now on, we assume that $W$ is spanned by the last $n-2$ vectors of the canonical basis of $V$. The next result gives a parametrisation and counts the number of subgroups with such a property. 

\begin{proposition}\label{cor_1}
Let $W \deq \Span{e_i \mid  3 \leq i \leq n}$.
The group $\Sym(V)$ contains $2^{n-2}-1$ elementary abelian regular subgroups $T_{b}$, where   $b\in W\setminus\Set{0}$, such that $T \cap T_b = \sigma_W$. More precisely, $T_{b} = \Span{\pi_b,\ve_b, \sigma_{e_i} \mid 3 \leq i \leq n}$, where
\begin{equation}\label{eq_pive}
\pi_b=
\left(
 \begin{array}{c|c}
 \one_{2} & 
 \begin{array}{c}
0\\
b^{(3:n)}
 \end{array}
 \\ 
 \hline
0& \one_{n-2}
 \end{array} 
 \right)\sigma_{e_1}, \quad
\ve_b = 
\left(
 \begin{array}{c|c}
 \one_{2} & 
 \begin{array}{c}
b^{(3:n)}\\
0
 \end{array}
 \\ 
 \hline
 0 & \one_{n-2}
 \end{array} 
 \right)\sigma_{e_2}.
 \end{equation}
\end{proposition}
\begin{proof}
As in the proof of Theorem~\ref{thm_TinAGL} we necessarily have  $\pi_b$ and $\ve_b$ to be defined by Eq.~\eqref{eq_pive}. 
Note that $b$ is completely determined by the fact that $x\pi+x+0\pi$ is a scalar multiple of $b$ for any $\pi \in T_b$.
 Hence the desired result follows from the fact that such subgroups are in one-to-one correspondence with the possible choices of $b \in W \setminus\Set{0}$.
\end{proof}

The general problem of parametrising all the elementary abelian regular subgroups $\bar{T}$ of $\Sym(V)$ and of $\AGL(V)$ according to the size of their intersection with $T$ is not easy in general.  Proposition~\ref{cor_1} solves this problem in the case of second-maximal intersection subgroups. Partial results have been obtained in the case $\bar{T} < \AGL(V)$ \cite{Calderini2017, Civino2018, Brunetta2017}. In~\cite{Calderini2017} a result similar to the following corollary is proved, where $\AGL(V)$ appears in place of $\Sym(V)$. The present form is a consequence of Theorem~\ref{thm_TinAGL} and the result is easily derived by Proposition~\ref{cor_1}.

\begin{corollary}\label{cor_tn}
The group $\Sym(V)$ contains $t_n$ elementary abelian regular subgroups whose intersection with $T$ is a second-maximal subgroup of $T$, where 
\[
t_n \deq \frac{\left(2^{n-2}-1\right)\left(2^{n-1}-1\right)\left(2^n-1\right)}{3}.
\]
\end{corollary}
\begin{proof}
The integer $t_n$ may be obtained as the product of $2^{n-2}-1$ and $(2^n-1)(2^n-2)/6$, respectively the number of elementary abelian regular subgroups which intersect $T$ in the subspace spanned by the last $n-2$ vectors of the canonical basis and the number of $(n-2)$-dimensional subspaces of $V$. 
\end{proof}


\section{Sylow $2$-subgroups of $\AGL(V)$}\label{sec_syl}

In this section, $S$ will denote a Sylow $2$-subgroup of $\AGL(V)$. Up to a conjugation with an element in $\AGL(V)$ we can assume
\begin{equation}\label{eq_syl}
S \deq 
\Set{
U\sigma_v \mid U \in \mathcal U, v \in V
}
\end{equation}
where $\mathcal U$ is the group of upper unitriangular matrices, a Sylow $2$-subgroup of $\GL(V)$.  
 Notice that $\mathcal U$ is generated by the matrices $1_{n}+E_{i,i+1}$, where  $1\leq i\leq n-1$ and $E_{i,j}$ is the matrix whose entries are all zero except the $(i,j)$-th entry which is 1.
\begin{remark}\label{rem:flag}
The action by conjugation of $S$ on $T$ is canonically identified as the action of $\mathcal U$ on $V$. As a $\mathcal U$-module, $V$ is uniserial, i.e.\ any $\mathcal U$-submodule of $V$ belongs to the \emph{maximal flag} $\Set{0} = V_0 < V_1 < \ldots < V_n = V$, where $V_i = \Span{ e_{n-i+1},\ldots,e_{n}}$ and $1 \leq i \leq n$. Conversely, for each given maximal flag $\mathcal F$ whose members are
\[\Set{0}= V_0 < V_1
< \ldots < V_n=V, \] if $\mathcal U$ is the stabiliser of $\mathcal F$ in $\GL(V)$, then $\mathcal U T$ is a Sylow $2$-subgroup of $\AGL(V)$. The previous construction yields a one-to-one correspondence between Sylow $2$-subgroups of $\AGL(V)$ and the set of maximal flags of subspaces of $V$. Indeed, given a maximal flag $V_0 < V_1 < \ldots < V_n$, the corresponding Sylow $2$-subgroup is exactly the stabiliser by conjugation of $\Set{\one_{\Sym(V)}}=\sigma_{V_0}<\sigma_{V_1}<\ldots<\sigma_{V_n}=T$. This fact is used throughout this section without any further reference.
\end{remark}
The following theorem is crucial in this section.
\begin{theorem}\label{prop_Syl}
Every Sylow $2$-subgroup $\Sigma$ of $\AGL(V)$  contains exactly one elementary abelian regular subgroup $T_\Sigma$ intersecting $T$ in a second-maximal subgroup of $T$ and which is normal in $\Sigma$.
\end{theorem}

\begin{proof}
Since $T \isnorm \AGL(V)$, then $T$ is a normal subgroup of every Sylow $2$-subgroup of $\AGL(V)$. Without loss of generality we can assume $\Sigma$ to be the Sylow $2$-subgroup $S$ defined in Eq.~\eqref{eq_syl}.
Let us now define $W \deq \Span{e_i \mid 3 \leq i \leq n}$  and a family of groups 
$\Theta \deq \Set{T_{b} \mid b \in W\setminus\{0\} }$, where 
\[
T_b \deq \Span{\pi_b,\ve_b,
 \sigma_W} < S.
\]
and $\pi_b$ and $\ve_b$ are as in the statement of Proposition~\ref{cor_1}. Since $\Size{\Theta}$ is odd and $S$ fixes $T$ by conjugation, then there exists $b \in W\setminus\Set{0}$ such that $S$ fixes $T_b$ by conjugation, i.e.\ $T_b \isnorm S$. It remains to be proven that such a group is unique. 
Now we prove that  the unique $b \in W$ that corresponds to a normal subgroup in $S$ is $(0,0,\ldots,0,1)$, from which the desired result follows.
First, by Remark~\ref{rem_TinAGLp}, we have $T< N_{\Sym(V)}(T_b)$ for every $b\in W$, so each $\sigma_v\in T<S$ normalises $T_b$. Moreover, any upper unitriangular matrix  in $\mathcal{U}$ fixes by conjugation $\sigma_W$, indeed if $U\in\mathcal{U}$ then $\sigma_{e_{i}}^U=\sigma_{e_{i}U}$ and $e_i U\in W$ for every $3 \leq i \leq n$. Therefore we are left to determine $b$ such that $\theta^{U}\in T_b$, for every $\theta\in\Set{\pi_b,\ve_b,\pi_b\ve_b}$ and $U\in\mathcal{U}$. 
Note that  $\theta \in T_b$ implies that 
$
x +x\theta+0\theta = \alpha b,
$
for some scalar $\alpha \in \FF$.
Now, for $\theta = \pi_b$, 
\[
\begin{array}{rl}
x\theta^U=xU^{-1}\pi_b U & = (xU^{-1} + (xU^{-1})^{(2)} b + e_1)U\\
& = xU^{-1}U + (xU^{-1})^{(2)} bU + e_1U\\
& = x + (xU^{-1})^{(2)} bU + e_1U.
\end{array}
\]
Hence $\pi_b \in T_b$ implies
$x+x\pi_b+0\pi_b = (xU^{-1})^{(2)} bU = \alpha b$.
As a consequence $b$ is a common eigenvector for the elements of $\mathcal U$, therefore $b=(0,0,\ldots,0,1)$.
\end{proof}

The previous theorem has the following converse. The same notation is used.

\begin{proposition}\label{rmk_TinS}
If  $\bar{T}$ is an elementary abelian regular subgroup of $\AGL(V)$ such that $\Size{\bar{T}\cap T}=2^{n-2}$, then there exists a Sylow $2$-subgroup $\Sigma$ of $\AGL(V)$  such that $\bar{T}=T_\Sigma\isnorm \Sigma$. 
\end{proposition}

\begin{proof}
Up to conjugation we may assume that $T \cap \bar{T} = \Span{e_i \mid  3 \leq i \leq n}$.
By Proposition~\ref{cor_1}, there exists $b\in W$ such that $\bar{T}=T_{b}$, where $W<V$ is the subspace defined by $\sigma_W=\bar{T}\cap T$. Choose a basis  $\Set{\bar e_1,\ldots, \bar{e}_n}$ of $V$ such that $\bar{e}_n=b$ and $W=\Span{\bar{e}_3,\ldots, \bar{e}_n}$. Let $L\in GL(V)$ be the linear map $L\colon e_i \mapsto \bar{e}_i$.  Clearly $\bar{T}=T_S^L$ is normal in $\Sigma\deq S^L$.
\end{proof}

We prove in Theorem~\ref{thm_314} that if a Sylow $2$-subgroup $\Sigma$ of $\AGL(V)$ contains a conjugate in $\Sym(V)$ of $T$ as a normal subgroup, then such a subgroup is either $T$ or $T_\Sigma$, where $T_\Sigma$ is as in Theorem~\ref{prop_Syl}. In order to do that, the following result is helpful.

\begin{theorem}\label{thm_314_1}
Let $\Sigma$ be a Sylow $2$-subgroup of $\AGL(V)$ and $\Set{0} = V_0 < V_1 < \ldots < V_n = V$ be the associated invariant flag as in Remark~\ref{rem:flag}. There exist $2^{d\binom{n-d}{2}}$ subgroups $T^g\le \Sigma$, where $g\in\Sym(V)$,  such that $ \sigma_{V_d} \leq T\cap T^g$ and $T\le N_{\Sym(V)}(T^g) = \AGL(V)^g$.
\end{theorem}

\begin{proof}
We shall use the canonical embedding of $\AGL(V)$ in $\GL\left(\FF^{\,n+1}\right)$ sending the affinity $\phi\colon x\mapsto xL+v$ into the linear map represented by the matrix
\[
\left (\begin{array}{c|c}
1 & v \\ \hline 0 & L
\end{array}\right).\] The action of the affinity $\phi$ can be recovered by the equality
\[
(1,x\phi) =(1,x)\left (\begin{array}{c|c}
1 & v \\ \hline 0 & L
\end{array}\right).
\]
Under this monomorphism the elements of $T$ are represented by the matrices in which $L=1_n$.
Assume now that $W=V_d$ is defined as in Remark~\ref{rem:flag} and that, with the same notation, $U=U_d\deq \Span{ v_1,\ldots, v_{n-d}}$ so that $V=U\oplus W$. Taking into account this decomposition, the matrices representing elements  of $\AGL(V)$ can be written in the block form 
\[ 
M\deq \left (\begin{array}{c|c|c}
1 & u & w \\ \hline 0 & A & B \\ \hline  0 & D & C
\end{array}\right)
\]
where $u\in U$ and $w\in W$ (each of them referred to the relevant bases).
Under this notation the elements of $\Sigma$ are represented by matrices of the form 
\[ 
X \deq \left (\begin{array}{c|c|c}
1 & u & w \\ \hline 0 & A & B \\ \hline  0 & 0 & C
\end{array}\right),\]
where  $A$ and $C$ are upper unitriangular matrices, i.e.\ $X$ is an upper unitriangular matrix of $\GL(V)$. Finally the matrices representing the elements $\sigma_z$ in $\sigma_W$, where $z\in W$,   are those of the form
\[
Y_z \deq \left (\begin{array}{c|c|c}
1 & 0 & z \\ \hline 0 & 1_{n-d} & 0 \\ \hline  0 & 0 & 1_d
\end{array}\right).
\]
Now we look for matrices $X$ commuting with $Y_z$ independently of the choice of $z$. The condition $XY_z=Y_zX$ gives $zC=z$ for all $z\in W$, so that $C=1_d$. 
We have shown that \[C_\Sigma(\sigma_W)=\Set{
\left (\begin{array}{c|c|c}
1 & u & w \\ \hline 0 & A & B \\ \hline  0 & 0 & 1_d
\end{array}\right)
\mid A \text{ upper unitriangular}, u\in U \text{ and } w\in W
}.\]
Next, we determine the transitive elementary abelian subgroups of $C_\Sigma(\sigma_W)$ containing $\sigma_W$. Let $\bar{T}$ be one of them, it is well known that $\bar{T}$ has to be a regular permutation group. Let 
\[
\bar X=\left (\begin{array}{c|c|c}
1 & u & w \\ \hline 0 & A & B \\ \hline  0 & 0 & 1_d
\end{array}\right) \in \bar{T}\setminus \sigma_W.
\]
Note that $\bar{X}$ represents an affinity sending $0$ to $(u,w)$. Up to the right multiplication by $Y_w$ we can restrict our attention to the case when $w=0$. In this case, the regularity implies that \[
\bar X_{(u,0)}\deq \left (\begin{array}{c|c|c}
1 & u & 0 \\ \hline 0 & A_u & B_u \\ \hline  0 & 0 & 1_d
\end{array}\right)\]
 is uniquely determined by the image $u$ of $0$ under the affinity represented by $\bar{X}_{(u,0)}$. This gives rise to an isomorphism from $U\oplus W$ onto $\bar{T}$ defined by \[
(u,w)\mapsto \bar{X}_{(u,w)} \deq \left (\begin{array}{c|c|c}
1 & u & w \\ \hline 0 & A_u & B_u \\ \hline  0 & 0 & 1_d
\end{array}\right)
\]  
where $A_{uA_{u'}+u'}=A_uA_{u'}$ and $B_{uA_{u'}+u'}=B_u+A_uB_{u'}$ for all $u,u'\in U$.
Since $\sigma_W$ is central in $\bar{T}$, the requirement that $\bar{T}$ is abelian is equivalent to the condition $\bar{X}_{(u,0)}\bar{X}_{(u',0)}=\bar{X}_{(u',0)}\bar{X}_{(u,0)}$ for all $u, u'\in U$. This in turn is equivalent to the following set of conditions
\begin{align}
u'+uA_{u'}&=u+u'A_{u} \label{eq:comm_A}\\
uB_{u'}&=u'B_{u} \label{eq:comm_B} \\
A_{u}A_{u'}&=A_{u'}A_{u}	 \label{eq:comm_3}\\
A_{u}B_{u'}+B_{u}&=A_{u'}B_{u}+B_{u'}  \label{eq:comm_4}
\end{align}
for all $u,u'\in U$. 
If we furthermore assume  that $\bar{T}$ is normalised by $T$ we have 
\begin{align*}
\bar{T} \ni \left (\begin{array}{c|c|c}
1 & u' & 0 \\ \hline 0 & 1_{n-d} & 0 \\ \hline  0 & 0 & 1_d
\end{array}\right)^{-1}
&\bar X_{(u,0)}
 \left (\begin{array}{c|c|c}
1 & u' & 0 \\ \hline 0 & 1_{n-d} & 0 \\ \hline  0 & 0 & 1_d
\end{array}\right) \\
 =\left (\begin{array}{c|c|c}
1 & u+u'+u'A_{u} & u'B_{u} \\ \hline 0 & A_{u} & B_{u} \\ \hline  0 & 0 & 1_d
\end{array}\right)
=
&\bar X_{(u,0)} \left (\begin{array}{c|c|c}
1 & u'+u'A_{u} & u'B_{u} \\ \hline 0 & 1_{n-d} & 0 \\ \hline  0 & 0 & 1_d
\end{array}\right).
\end{align*}
It follows that 
\[\left (\begin{array}{c|c|c}
1 & u'+u'A_{u} & u'B_{u} \\ \hline 0 & 1_{n-d} & 0 \\ \hline  0 & 0 & 1_d
\end{array}\right) \in T\cap \bar{T} = \sigma_W,
\]
so that $u'+u'A_{u} = 0$ for all $u\in U$ and $u'\in U$, i.e.\ $A_{u}=1_{n-d}$ for all $u\in U$.
As a consequence Eq.~\eqref{eq:comm_A}, \eqref{eq:comm_B}, \eqref{eq:comm_3}, and \eqref{eq:comm_4} reduce to Eq.~\eqref{eq:comm_B}, and the map $u\mapsto B_u$ is linear.
Finally, $\bar{T}$ is elementary abelian if and only if $\bar{X}_{(u,0)}^2=1_{n+1}$ for all $u\in U$. This final condition can be stated equivalently in the form 
\begin{equation}\label{eq:exp2}
uB_u=0 \text{ for all $u\in U$.}
\end{equation}
The group $\bar{T}$ is then uniquely determined by the linear map $u\mapsto B_u$, which in turn is defined once for all $1 \leq i \leq n-d$ the matrix $B_{e_i}$ is given. Such matrices have to satisfy Eq.~\eqref{eq:comm_B} and of Eq.~\eqref{eq:exp2}, i.e.\ the $i$-th row of $B_{e_j}$ is equal to the $j$-th row of $B_{e_i}$ and  the $i$-th row of $B_{e_i}$ is the zero row for $1\le i , j \le n-d$. Thus, the total number of possible choices for $\bar{T}$ is easily seen to be $2^{d\binom{n-d}{2}}$.
\end{proof}

\begin{theorem}\label{thm_314}
Let $g \in \Sym(V)$ and let $\Sigma$ be a Sylow $2$-subgroup of $\AGL(V)$ containing $T^g$. The subgroup $T^g$ is normal in $\Sigma$ if and only if $T^g \in \Set{T, T_\Sigma}$.
\end{theorem}
\begin{proof}
We shall use the same notation as in the proof of the previous theorem. If $T^g$ is normal in $\Sigma$ then $T\cap T^g=\sigma_{V_d}$ for some $1\le d\le n$, $d\ne n-1$.  Notice that $T^g < \Sigma$, so $T$ normalises ${T}^g$. Thus
\[
T^g=\bar{T}= \Set{ \left(\begin{array}{c|c|c}
1 & u & w \\ \hline 0 & 1_{n-d} & B_u \\ \hline  0 & 0 & 1_d
\end{array}\right) \mid u\in U }\]
is determined by the linear map $u\mapsto B_u$. Consider the matrices 
\[
M_{i,j}\deq \left (\begin{array}{c|c|c}
1 & 0 & 0 \\ \hline 0 & 1_{n-d} & 0 \\ \hline  0 & 0 & 1_d +E_{i,j}
\end{array}\right),
\] where $E_{i,j}$ is defined in the beginning of Sec.~\ref{sec_syl}. These matrices represent elements of $\Sigma$ such that $M_{i,j}=M_{i,j}^{-1}$. A direct computation shows that 
\begin{gather*}
M_{i,j}^{-1}\left(\begin{array}{c|c|c}
1 & u & w \\ \hline 0 & 1_{n-d} & B_u \\ \hline  0 & 0 & 1_d
\end{array}\right)M_{i,j}= 
\left(\begin{array}{c|c|c}
1 & u & w(1_{d}+E_{i,j}) \\ \hline 0 & 1_{n-d} & B_u(1_{d}+E_{i,j}) \\ \hline  0 & 0 & 1_d
\end{array}\right)
\end{gather*}
so that if $\bar{T}$ is normal in $\Sigma$ then $B_u(1_{d}+E_{i,j}) = B_u$ for every $u\in U$ and $1 \le i < j \le d$. It follows that, for $1 \leq i \leq d-1$, the $i$-th column of $B_u$ is the zero column. The matrix 
\[N_{i,j} \deq \left (\begin{array}{c|c|c}
1 & 0 & 0 \\ \hline 0 & 1_{n-d}+E_{i,j} & 0 \\ \hline  0 & 0 & 1_d
\end{array}\right),
\] where $1\le i < j \le n-d$, represents an involution in $\Sigma$ and 
\begin{gather*}
N_{i,j}^{-1}\left(\begin{array}{c|c|c}
1 & e_h & w \\ \hline 0 & 1_{n-d} & B_{e_h} \\ \hline  0 & 0 & 1_d
\end{array}\right)N_{i,j}= 
\left(\begin{array}{c|c|c}
1 & e_h(1_{n-d}+E_{i,j}) & w \\ \hline 0 & 1_{n-d} & (1_{n-d}+E_{i,j})B_{e_h} \\ \hline  0 & 0 & 1_d
\end{array}\right).
\end{gather*}
If this matrix is in $\bar{T}$ and $h\ne i$, then $ e_h(1_{n-d}+E_{i,j}) = e_h$, so that $(1_{n-d}+E_{i,j})B_{e_h}=B_{e_h}$, which in turn implies that the $j$-th row of $B_{e_h}$ is the zero row for every $j\ne i$ and $j\ne h$. Hence $B_{e_h}=0$ for $3\leq h \leq n$, so $\Span{e_3,\ldots,e_n} \leq W$ and thus $d  = \dim(W)\geq n-2$. Moreover, by Proposition~\ref{six} we have $d \leq n-2$, and therefore $d=n-2$. We can now apply Theorem~\ref{prop_Syl} to obtain that $\bar{T}\isnorm \Sigma$ if and only if $T^g=\bar{T}\in \Set{T, T_\Sigma}$.
\end{proof}

With the same notation of the previous theorems we have the following corollary.
\begin{corollary}\label{cor:Tg=Tp}
Every $g\in N_{\Sym(V)}(\Sigma)\setminus\AGL(V)$ interchanges by conjugation $T$ and $T_\Sigma$. Moreover each element in $ \AGL(V)\cap  N_{\Sym(V)}(T_\Sigma)=\AGL(V)\cap  \AGL(V)^g$ stabilises this action.
\end{corollary}
\begin{proof}
Since the normaliser $N_{\Sym(V)}(\Sigma)$ permutes the normal elementary abelian regular subgroups of $\Sigma$ and $\AGL(V)$ normalises $T$, the claim follows straightforwardly.
\end{proof}

From Theorem~\ref{thm_314} we can also conclude that second-maximal intersection subgroups are characterised by the fact that  their normalisers contain a Sylow $2$-subgroup of $\AGL(V)$.
\begin{corollary}
Let $g \in \Sym(V)\setminus \AGL(V)$ such that $T^g$ is an elementary abelian regular subgroup of $\Sym(V)$. If $\Size{\AGL(V)}=2^{m}t$, with $t$ an odd integer, then
\[
\Size{T \cap T^g} = 2^{n-2} \iff 2^m \,\big{\mid} \Size{\AGL(V) \cap \AGL(V)^g}.
\] 
\end{corollary}
\begin{proof}
If $\Size{T \cap T^g} = 2^{n-2}$, then by Theorem~\ref{thm_TinAGL}, both $T$ and $T^g$ are subgroups of $\AGL(V)$. Moreover, since $T$ is contained in every Sylow $2$-subgroup of $\AGL(V)$, at least one of them, which we denote by $\Sigma$, contains both $T$ and $T^g$ as normal subgroups. Thus $\Sigma \leq \AGL(V) \cap \AGL(V)^g$. Conversely, if this is the case, $T$ and $T^g$, being contained in every Sylow $2$-subgroup of their own normalisers, are distinct normal subgroups of $\Sigma$. Therefore, Theorems~\ref{prop_Syl} and \ref{thm_314} yield
$\Set{T, T^g}=\Set{T, T_\Sigma}$, hence $\Size{T\cap T^g}=2^{n-2}$.
\end{proof}

It was already known to P. Hall (see e.g.~\cite{Carter1964}) that if $\Xi$ is a Sylow $2$-subgroup of $\Sym(V)$, then $N_{\Sym(V)}(\Xi)=\Xi$. In the remainder of the paper we establish a similar result for Sylow $2$-subgroups of $\AGL(V)$. 

\begin{theorem}\label{thm:N=S}
If  $\Sigma$ is a Sylow $2$-subgroup of $\AGL(V)$, then 
\[
[N_{\Sym(V)}(\Sigma):\Sigma]=2.
\]
\end{theorem}
\begin{proof}
By Remark~\ref{rem:flag} there exists a flag $\Set{0}= V_0 < V_1
< \ldots < V_n=V$ such that $\Sigma$ is the stabiliser by conjugation of $\sigma_{V_0}<\sigma_{V_1}<\ldots<\sigma_{V_n}=T$.
By Theorem~\ref{thm_314}, for every $g\in N_{\Sym(V)}(\Sigma)$, we have $\sigma_{V_i}^{g^2}=\sigma_{V_i}$ for $1 \leq i \leq n$, and so $g^2\in \Sigma$. 
Hence $N_{\Sym(V)}(\Sigma)/\Sigma$ has exponent equal to 2, i.e.\ it is an elementary abelian group. Let now $n,m\in N_{\Sym(V)}(\Sigma)\setminus \Sigma$. We have $T^{\,nm^{-1}}=T$, and so $nm^{-1}\in \Sigma$. Therefore $N_{\Sym(V)}(\Sigma)/\Sigma$ contains only one non-trivial coset, and consequently it has order 2.
\end{proof}

\begin{remark}
If $\dim (V)=n>2$, then $\AGL(V)$ is a subgroup of $\Alt(V)$, the alternating group on $V$. 
Indeed, we note that a subgroup $H$ of the symmetric group is contained in the alternating group if and only if one of its Sylow $2$-subgroups is. The reason is that all the elements of odd order are contained in the alternating subgroup, and any other element is the product of a $2$-element by an element of odd order. 
In our case, all the elements of $T$ are the product of an even number of transpositions, so that $T<\Alt(V)$. Hence it suffices to show that any Sylow $2$-subgroup of $\GL(V)$  is contained in $ \Alt(V)$. We recall that the matrices $1_{n}+E_{i,i+1}$, defined in the beginning of Sec.~\ref{sec_syl}, generate a Sylow $2$-subgroup of $\GL(V)$. Each of such matrices fixes $2^{n-1}$ vectors, whereas it pairwise exchanges the remaining $2^n - 2^{n-1}$ ones, i.e., as a permutation, it is the product of  $2^{n-1}-2^{n-2}$ of transpositions, which is an even number. 
\end{remark}

\begin{corollary}
The normaliser in $\Sym(V)$ of each Sylow $2$-subgroup $\Sigma$  of $\AGL(V)$ consists of even permutations. Thus
$N_{\Sym(V)}(\Sigma)=N_{\mathrm{\Alt(V)}}(\Sigma)$.
\end{corollary}
\begin{proof}
Let $\Sigma$ be a Sylow $2$-subgroup of $\AGL(V)$. If $\Xi$ is a Sylow $2$-subgroup of $\Alt(V)$ containing $\Sigma$, then $\Xi>\Sigma$, so that $N_{\Xi}(\Sigma)>\Sigma$. By Theorem~\ref{thm:N=S}, it follows that $N_{\Alt(V)}(\Sigma)\geq N_{\Xi}(\Sigma)=N_{\Sym(V)}(\Sigma)$.
\end{proof}
 The result which follows is the counterpart in $\AGL(V)$ of the result due to P. Hall on the Sylow $2$-subgroups of $\Sym(V)$. It also allows us to count the number of distinct Sylow $2$-subgroups of $\AGL(V)$.
\begin{theorem}
If $\Sigma$ is Sylow $2$-subgroup of $\AGL(V)$, then $N_{\AGL(V)}(\Sigma)=\Sigma$. In particular, 
\begin{equation}\label{eq:numberS}
[\AGL(V):\Sigma]=\prod_{j=0}^{n-1}{\left(2^{n-j}-1\right)}.
\end{equation}
is the number of distinct Sylow $2$-subgroups of $\AGL(V)$.
\end{theorem}
\begin{proof}
By Theorem~\ref{thm:N=S}, if $\Size{\AGL(V)}=2^m t$, with $t$ an odd integer, then we have $\Size{\Sigma}=2^m$ and $\Size{N_{\Sym(V)}(\Sigma)}=2^{m+1}$. Since $N_{\AGL(V)}(\Sigma)\leq\AGL(V)$ and $N_{\AGL(V)}(\Sigma) \leq N_{\Sym(V)}(\Sigma)$, then $\Size{N_{\AGL(V)}(\Sigma)}=2^m$.
\end{proof}

Finally, we use Theorem~\ref{thm_314} to give an alternative proof of the following result. 
\begin{corollary}[\cite{Calderini2017}]
The group $\AGL(V)$  acts transitively by conjugation on the set of elementary abelian regular subgroups which intersect $T$ in a second-maximal subgroup of $T$.
\end{corollary}
\begin{proof}
It is enough to show that each elementary abelian regular group $\bar T$ such that $\Size{T \cap \bar T} = 2^{n-2}$ is conjugated to $T_S$, where $S$ is the Sylow $2$-subgroup of $\AGL(V)$ defined as in Eq.~\eqref{eq_syl}. By Proposition~\ref{rmk_TinS} there exists a Sylow $2$-subgroup  $\bar S$ of $\AGL(V)$  such that $\bar T \isnorm \bar S$.  Moreover $\bar{S}^h = S$ for some $h\in \AGL(V)$. By Theorem~\ref{thm_314} it follows that $T_S=\bar{T}^h\isnorm \bar{S}^h$.
\end{proof}
As a last consequence, the number of Sylow $2$-subgroups of $\AGL(V)$ which contain the same second-maximal-intersection subgroup as a normal subgroup can be determined.
\begin{corollary}
The number $s_n$ of Sylow $2$-subgroups of $\AGL(V)$ which contain as a normal subgroup the same group $T^g$ such that $\dim(W)=n-2$, where $\sigma_W= T \cap T^g$ and $g \in \Sym(V)$, is given by the formula:
\[
s_n=3\prod_{j=3}^{n-1}{\left(2^{n-j}-1\right)}.
\]
\end{corollary}

\begin{proof}
Let $\Size{\AGL(V)}=2^{m}t$, with $t$ an odd integer. First we recall that $t$ is the integer displayed in Eq.~\eqref{eq:numberS}. The claim follows from Corollary~\ref{cor_tn}, since $s_n={t}/{t_n}$, where $t_n$ is the number of elementary abelian regular subgroups in $\AGL(V)$ whose intersection with $T$ is a second-maximal subgroup of $T$.
\end{proof}


\section{Conclusion and open problems}\label{sec_concl}
We already mentioned that the conjugates in $\Sym(V)$ of $T$ are very important in the cryptanalysis of block ciphers. For this reason, a complete parametrisation of them
in terms of the size of their intersection with $T$ is needed. Recall that the elements of such intersections are in one-to-one correspondence with the weak keys corresponding to the alternative operations. 
In this paper, the aforementioned problem has been addressed, both considering subgroups of $\Sym(V)$ and $\AGL(V)$, in the case where the weak-key subspace has dimension $n-1$ and $n-2$. 
This last case turns out to be one of the most relevant for cryptanalysis, for  reasons whose description would lead us out of the scope of this work. 
We have computational evidence that also the case of lower dimensional weak-key spaces might be interesting from a cryptographic point of view, though it may require an entirely different  technical approach. 


\bibliographystyle{amsalpha}
\bibliography{T_Pallino_rev_1}

\end{document}